\documentclass[11pt]{amsart}
\usepackage{graphicx}
\usepackage[active]{srcltx}
 \makeatletter
\renewcommand*\subjclass[2][2000]{%
  \def\@subjclass{#2}%
  \@ifundefined{subjclassname@#1}{%
    \ClassWarning{\@classname}{Unknown edition (#1) of Mathematics
      Subject Classification; using '1991'.}%
  }{%
    \@xp\let\@xp\subjclassname\csname subjclassname@#1\endcsname
  }%
}
 \makeatother
\usepackage{enumerate,amssymb,  mathrsfs,yhmath}

\newtheorem{theorem}{Theorem}[section]
\newtheorem{lemma}[theorem]{Lemma}
\newtheorem*{lemma*}{Lemma}
\newtheorem{proposition}[theorem]{Proposition}
\newtheorem{corollary}[theorem]{Corollary}

\def\1ton{1,2,\ldots,n}

\def\Div{{\rm Div}}

\def\d{\textnormal d}

\usepackage{amssymb}
\usepackage{amsthm}
\usepackage{mathrsfs, amsfonts, amsmath}
\usepackage{graphicx}
\newcommand{\norm}{\,|\!|\,}
\newcommand{\bydef}{\stackrel{{\rm def}}{=\!\!=}}

\newcommand{\onto}{\xrightarrow[]{{}_{\!\!\textnormal{onto}\!\!}}}

\newcommand{\Tr}{\text{Tr}}
\newcommand{\A}{\mathbb{A}}
\newcommand{\R}{\mathbb{R}}

\newcommand{\X}{\mathbb{X}}

\newcommand{\Y}{\mathbb{Y}}
\newcommand{\W}{\mathscr{W}}

\theoremstyle{definition}

\newtheorem{conjecture}[theorem]{Conjecture}

\theoremstyle{remark}
\newtheorem{remark}[theorem]{Remark}

\numberwithin{equation}{section}

\def\XXint#1#2#3{{\setbox0=\hbox{$#1{#2#3}{\int}$}
\vcenter{\hbox{$#2#3$}}\kern-.5\wd0}}

\def\ge{\geqslant}
\begin{document}

\title{Harmonic maps between two concentric annuli in $\mathbf{R}^3$}  \subjclass{Primary 31A05;
Secondary 42B30 }


\author{David Kalaj}
\address{University of Montenegro, Faculty of Natural Sciences and
Mathematics, Cetinjski put b.b. 81000 Podgorica, Montenegro}
\email{davidk@ac.me}

\begin{abstract}
Given two annuli $\A(r,R)$ and $\A(r_\ast, R_\ast)$, in $\mathbf{R}^3$ equipped with the Euclidean metric and the weighted metric $|y|^{-2}$ respectively,   we minimize the Dirichlet integral, i.e. the functional $\mathscr{F}[f]=\int_{\A(r,R)}\frac{\norm Df\norm ^2} {|f|^2}$, where $f$ is a homeomorphism between $\A(r,R)$ and $\A(r_\ast,R_\ast)$, which belongs to the Sobolev class $\mathscr{W}^{1,2}$. The minimizer is a certain generalized radial mapping, i.e. a mapping of the form $f(|x|\eta)=\rho(|x|)T(\eta)$, where $T$ is a conformal mapping of the unit sphere onto itself and $\rho(t)={R_\ast} \left(\frac{r_\ast }{R_\ast }\right)^{\frac{R (r-t)}{(R-r) t}}.$ It should be noticed that in this case no Nitsche phenomenon occur.

\end{abstract}
\maketitle
\section{Introduction and statement of the main result}

The general law of hyperelasticity tells us that there exists an energy integral
$ E[h] = \int_\X
E(x, h, Dh) dx$  where $E : \mathbb{X} \times \mathbb{Y} \times \mathbf{R}^{n\times n}\to \mathbf{R}$ is a given stored-energy function characterizing
mechanical properties of the material. Here  $\X$ and $\Y$ are nonempty bounded domains in $\mathbf{R}^n
, n > 2.$
The mathematical models of nonlinear elasticity have been firstly studied by Antman \cite{[2]}, Ball \cite{[5]}, and Ciarlet \cite{[13]}. One of interesting and important problems in nonlinear
elasticity is whether the radially symmetric minimizers are indeed global
minimizers of the given physically reasonable energy. This leads us to study energy minimal
homeomorphisms $h: \A
\onto \A_\ast$ of Sobolev class $\mathscr{W}^{1,2}$ between annuli
$\A = \A(r, R) = \{x \in\mathbf{R}^n: r < |x| < R\}$ and $\A_\ast
=\A(r_\ast, R_\ast) = \{x \in\mathbf{R}^n: r_\ast < |x| < R_\ast\}$.
Here $0 \le  r < R$ and $0 \le  r_\ast < R_\ast$ are the inner and outer radii of $\A$ and
$\A_\ast$. The variational approach to Geometric Function Theory \cite{atm1,atm}  makes this problem more important. Indeed, several papers are
devoted to understand the expected radial symmetric properties see \cite{ko} and the references therein. Many times experimentally known
answers to practical problems has led us to deeper study of such mathematically
challenging problems.  We seek to minimize the 2-harmonic energy of mappings between two annuli in $\mathbf{R}^3$. We consider the modified Dirichlet energy $\mathscr{F}[f]=\int_{\A}\frac{\norm Dh\norm ^2}{|h|^2}$ and solve the problem of modified Dirichlet energy in the fourth section. The problem for  Dirichlet energy  $\mathscr{E}[f]=\int_{\A}\norm Dh\norm ^2$ is considered in the appendix below, but not solved completely. The research is related to the J. C. C. Nitsche conjecture \cite{Nitsche}. The conjecture has been solved by Kovalev, Iwaniec and Onninen in \cite{nconj} after some partial results by Lyzzaik \cite{Al}, Weitsman \cite{aw} and Kalaj \cite{Ka}. The conjecture raised a very important research in Geometric Function Theory connected to the nonlinear elasticity. See for example the papers \cite{atm}, \cite{memoirs} and \cite{klondon}.

In order to formulate the main result, let us define the generalized radial mappings.

We say that $f:\A\to \A_\ast$ is a generalized radial mapping, if there exists a conformal transformation $T$ of $\mathbb{S}$ onto itself, so that $f(x)=\rho(|x|)T\left(\frac{x}{|x|}\right)$. If $T$ is the identity, then we remove the prefix "generalized". For the representation of the class of conformal mappings of the sphere onto itself we refer to the books \cite{ah} and \cite{matti}.

The following is the main result of the paper

\begin{theorem}\label{krye}
Assume that $\mathcal{F}$ is the family a homeomorphisms between spherical rings $\A(r,R)$ and $\A(r_\ast, R_\ast)$ in $\mathbf{R}^3$ that belongs to $\mathscr{W}^{1,2}$. Then for the Dirichlet integral of $f\in\mathcal{F}$ with respect to the weight $\wp(w)=|w|^{-2}$, we have $$\mathscr{F}[f]=\int_{\A(r,R)}\frac{\norm Df\norm ^2}{|f|^2} dx\ge 4 \pi  \left(2(R-r)+\frac{r R \log \left[\frac{R_\ast}{r_\ast }\right]^2}{R-r}\right),$$ where $dx$ is the Lebesgue measure, and the infimum is achieved for the following generalized radial difeomorphisms between annuli $$f_1(x)={r_\ast} \left(\frac{r_\ast }{R_\ast }\right)^{\frac{R (r-|x|)}{(R-r) |x|}}T\left(\frac{x}{|x|}\right), \ \ \ f_2(x)={R_\ast} \left(\frac{r_\ast }{R_\ast }\right)^{\frac{R (|x|-r)}{(R-r) |x|}}T\left(\frac{x}{|x|}\right) .$$ The minimizer is unique up to a conformal change $T$ of $\mathbb{S}$.
\end{theorem}

\begin{remark}
 If we denote the outer boundary of $\A$ by $\partial_\circ \A$ and consider the subfamily of homomorphisms $\mathcal{F}_\circ=\{f\in\mathcal{F}: f(x)=\frac{R_\ast}{R} x, \  \text{ for } x\in \partial_\circ \A$, then the minimizer is the mapping $h(x)=\rho(x)\frac{x}{|x|}$. See the paper by Koski and Onninen \cite{ko} where they make this constraint in order to prove that the minimizer is radial but for annuli on the plane, and $p$ energy. On the other hand when $R_\ast=r_\ast=1$, then the result says that the mappings $h(x)=T(x/|x|)$, of the unit sphere onto itself minimize the energy of mappings onto the unit sphere. This is an old problem solved by several authors ( see for example \cite{bcl}, \cite{jcb}, \cite{hmc}). Theorem~\ref{krye}, together with its Corollary~\ref{rrje} says that the case of degeneric annuli ($r=r_\ast=0$) is substantially different from the case of proper annuli concerning the Dirichlet energy. In the case of degeneric annuli, the minimal energy is zero (\cite{advc}).
\end{remark}

\section{Harmonic mappings and $p-$harmonic mappings}
In the following we define several classes of mappings which appear as the critical points of various energy integrals.
Assume that $h=\varrho^2$ is a positive smooth real function defined in the domain $\A_\ast$. Then it defines the Riemannian manifold $(\A_\ast, h)$. Assume that $\A$ is equipped with the Euclidean metric $g=1$ and let $f:(\A,g)\to(\A_*,g)$ be a $C^1$ map between manifolds. The energy density is defined \cite[Chapter~IX]{sy} by
$$e(f)=\mathrm{Tr}_g(f^*h)=\sum_{\alpha,\beta,i,j}^ng^{\alpha,\beta}(x)h_{ij}(f(x))\frac{\partial u^i}{\partial x^\alpha}\frac{\partial u^j}{\partial x^\beta}.$$

Thus $$e(u)=\varrho^2(f(x))\sum_{\alpha,\beta,i,j}^n\frac{\partial u^i}{\partial x^\alpha}\frac{\partial u^j}{\partial x^\beta}=\varrho^2(f(x))\norm Dh \norm ^2,$$
where $\norm \cdot \norm $ is the Gram-Schmidt norm defined by $\norm Dh \norm ^2 =\mathrm{Tr} (Dh^\ast Dh).$
Assume that $2\le p\le n $ and let $\wp\bydef \varrho^p$. The classical Dirichlet problem concerns the energy minimal mapping $h \colon \mathbb{A} \to \mathbf{R}^n$ of the Sobolev class $h\in h_\circ + \W^{1,n}_\circ (\mathbb{A}, \mathbf{R}^n)$ whose boundary values are explicitly prescribed by means of a given mapping $h_\circ \in  \W^{1,n} (\mathbb{A}, \mathbf{R}^n)$. More precisely we deal with the energy integral $$\mathscr{E}_{p}[h]=\mathscr{E}_{\rho,p}[h]\bydef \int_{\A}e(f)^{p/2}dx=\int_{\A}\wp(h(x))\norm Dh \norm ^p dx.$$

Let us consider the variation   $h \leadsto h\,+ \,\epsilon \eta $,  in which $\eta \in \mathscr C^\infty_\circ (\mathbb{A} , \R^n)$ and $\epsilon \to 0$, leads to the integral form of the $p$-harmonic system of equations
\begin{equation}\label{equa1}
\int_{\mathbb{A}} \left(\left<\nabla \rho, \eta\right>\norm Dh\norm ^p+\langle \wp(h) \norm Dh\norm^{p-2}Dh \right), \, D\eta \rangle =0, \quad \mbox{ for every } \eta \in \mathscr C^\infty_\circ (\mathbb{A} , \R^n).
\end{equation}
Equivalently
\begin{equation}\label{equa2}
\Delta_p h = \Div \big( \wp(h)\norm Dh \norm^{p-2}Dh\big)-\frac{1}{p}\norm Dh\norm ^p\nabla \wp=0,
\end{equation} in the  sense of distributions. The solutions to \eqref{equa2} are called $p-$harmonic mappings.

If $p=2$ the equation is called the harmonic equation, and the solutions are called the harmonic mappings.

Similarly as in in \cite{memoirs} (see also \cite{arxivk}), it can be derived the general $(\wp,p)$-harmonic equation which by using a different variation as the following.

The situation is different if  we allow $h$ to slip freely along the boundaries. The {\it inner variation} come to stage in this case. This is simply a change of the  variable; $h_\epsilon=h \circ \eta_\epsilon $, where  $\eta_\epsilon \colon \A \onto \A$ is a $\mathscr C^\infty$-smooth diffeomorphsm  of $\A$ onto itself, depending smoothly on a parameter $\epsilon \approx 0$ where  $\eta_\circ = id \colon \A \onto \A$. Let us take on the inner variation of the form
\begin{equation}\label{equa7}
\eta_\epsilon (x)= x + \epsilon \, \eta (x), \qquad \eta \in \mathscr C_\circ^\infty (\A, \R^n).
\end{equation}
By using the notation  $y=x+\epsilon \, \eta (x) \in \A$,
we obtain
$$\wp(h_\epsilon)Dh_\epsilon (x) = \wp(h(y)) Dh (y) (I+ \epsilon D\eta(x)).$$ Hence
\[
\begin{split}
\wp(h_\epsilon(x))\norm Dh_\epsilon(x)\norm^p & = \wp(h(y))\norm Dh(y)\norm^p
\\&+ p \epsilon\,  \wp(h(y))\langle \norm{Dh(y)}\norm ^{p-2} D^\ast h(y)\cdot  Dh(y)\, ,\,  D \eta \rangle + o(\epsilon).
\end{split}
\]
Integration with respect to $x\in \A$ we obtain
\[\begin{split}\mathscr E_\rho[h_\epsilon] &=\int_\A \wp(h_\epsilon(x))\norm Dh_\epsilon(x)\norm^p dx\\&= \int_\A \bigg[ \wp(h(y))\norm Dh(y) \norm^p \\& \ \ \ \ +  p \epsilon \wp(h(y))\langle \norm{Dh(y)}\norm ^{p-2} D^\ast h(y)\cdot  Dh(y)\, ,\,  D \eta(x) \rangle \bigg]\, \d x + o(\epsilon)\end{split} .\]
We now make the substitution $y=x + \epsilon \, \eta (x)$, which is a diffeomorphism for small $\epsilon$, for which we have:  $x= y- \epsilon \, \eta (y)+ o(\epsilon)$, $D\eta (x)= D\eta (y)+o(1)$, when $\epsilon\to 0$, and the change of volume element $\d x = [1-\epsilon \, \Tr \,D \eta (y) ]\, \d y + o(\epsilon) $. Further
$$\int_\A \wp(h(y))\norm Dh(y) \norm^p \d x=\int_\A \wp(h(y))\norm Dh(y) \norm^p [1-\epsilon \, \Tr \,D \eta (y) ]\, \d y + o(\epsilon)$$
 The so called equilibrium equation for the inner variation is obtained from $\frac{\d}{\d \epsilon} \mathscr E_p[{h_\epsilon}]\,=\,0\,$ at $\epsilon =0$,
\begin{equation}\label{intstar}\int_\A \langle \wp(h)\norm Dh \norm^{p-2} D^\ast h \cdot Dh - \frac{\wp(h)}{p} \norm Dh \norm^p I \, , \, D \eta \rangle \, \d y=0 \end{equation}
or, by using distributions
\begin{equation}\label{enhe}
\Div \left(\wp(h)\norm Dh \norm^{p-2} D^\ast h \cdot Dh - \frac{\wp(h)}{p} \norm Dh \norm^p I  \right)=0.
\end{equation}

By putting $$h(x)=H(t)\frac{x}{t}, \ \  t=|x|$$ we get $$Dh(x)=\frac{H(t)}{t}\mathrm{I}+\frac{tH'(t)-H(t)}{t}\cdot \frac{x\otimes x}{|x|^2}$$ and $$ \norm Dh \norm^{2}=\dot H(t)^2+(n-1) \frac{H(t)^2}{t^2}$$

Then we obtain $$D^\ast h \cdot Dh=\frac{H(t)^2}{t^2}\mathrm{I}+ \frac{t^2\dot H(t)^2-H(t)^2}{t^2} \frac{x\otimes x}{|x|^2}$$

We will focus on a particular problem, i.e. the case $n=3$, $p=n-1=2$ and $\wp(y)=|y|^{-2}$. So we consider the harmonic mappings between threedimensional Riemannian manifolds $(\A,g)$ and $(\A_\ast, h)$.

Then we have
\[\begin{split}\wp(h)\norm Dh \norm^{p-2} D^\ast h \cdot Dh &- \frac{\wp(h)}{p} \norm Dh \norm^p I=\wp(h)\left( D^\ast h \cdot Dh - \frac{1}{2} \norm Dh \norm^2 I\right)\\&=\left(-\frac{\dot H(t)^2}{2H^2(t)}\mathrm{I}+ \frac{t^2\dot H(t)^2-H(t)^2}{t^2H^2(t)} \frac{x\otimes x}{|x|^2}\right)\\&=(M(t)-t^{-2})\frac{x\otimes x}{|x|^2}-\frac{M(t)}{2}\mathrm{I},\end{split}\] where \begin{equation}\label{mh}M(t)=\frac{\dot H(t)^2}{H^2(t)}, \ \   t=|x|,  \ \  x=(x_1,x_2,x_3).\end{equation}

Now \eqref{enhe}  reduces to the differential equation $$ \left(\frac{2 M(t)}{t}+\frac{M'(t)}{2 }\right) \frac{x}{|x|}= 0.$$
By having in the mind the substitution \eqref{mh} we obtain the following equation

\begin{equation}\label{bubi}\left(\frac{2 H(t) \dot H(t)- t \dot H(t)^2+t H(t) \ddot H(t)}{t^2 H(t)}\right)x\equiv 0.\end{equation}

In order to consider the equation
\eqref{equa2} for the case $n-1=2=p$, we first have $$\Div \left( \frac{Dh}{|h|^2}\right)=\frac{1}{2}\norm Dh\norm ^2\nabla \rho=-\frac{\norm Dh\norm ^2}{|h|^4}h.$$

Then \begin{equation}\label{hh}\Delta h = \frac{2}{|h|^2} \sum_{j=1}^3\sum_{k=1}^3 D_k h_j \left<h, D_k h\right>e_j-\frac{\norm Dh\norm ^2}{|h|^2}h.\end{equation}

Put in the previous equation $h(x)=H(t)\frac{x}{|x|}$, where $t=|x|$. Then we have

\begin{equation}\label{delta}\Delta h=\frac{-2 H(t)+2 t H'(t)+t^2 H''(t)}{t^3} x$$ and $$\norm Dh\norm ^2=\frac{2 H(t)^2}{t^2}+H'(t)^2\end{equation} and

$$ \frac{2}{|h|^2} \sum_{j=1}^3\sum_{k=1}^3 D_k h_j \left<h, D_k h\right>e_j=\frac{2}{H(t)^2}\frac{ H(t) H'(t)^2}{t}x$$
By plugging the previous three quantities in \eqref{hh} we get again \eqref{bubi}.

It follows from our main result that if instead of $h(x)=H(t)\frac{x}{|x|}$, we  put the following constraint $h(x)=H(t)T\left(\frac{x}{|x|}\right)$ in \eqref{hh} we again arrive to the following equation
 \begin{equation}\label{bubi1}\left(\frac{2 H(t) \dot H(t)- t \dot H(t)^2+t H(t) \ddot H(t)}{t^2 H(t)}\right)T\left(\frac{x}{|x|}\right)\equiv 0,\end{equation} which is equivalent to \eqref{bubi}. We will solve those equations later.

It is easily seen that one of the solution of \eqref{bubi} is induced by the function $H(t)=1$, namely the mapping $h(x)=\frac{x}{|x|}$. This mapping is harmonic and solves both equations \eqref{equa2} and \eqref{enhe} but it is not a diffeomorphsim. This makes a substantial difference between the corresponding equations in \cite{memoirs}, where the authors Iwaniec and  Onninen shown that the mapping $f(x)=\frac{x}{|x|}$ is generalized $n-$harmonic but it is not $n-$harmonic.

\section{Some preliminary results}
For a mapping $f\in \mathcal{F}(\A,\A_\ast)$ we put
$$f=\rho(x) S(x), \ \   |S(x)|=1.$$
Then
$$Df(x)=\nabla \rho(x) \otimes S(x)+\rho DS(x).$$
So for any vector $k$ we have $$Df(x)k =\left<\nabla \rho (x) , k \right> S(x)+\rho DS(x) k .$$
It follows that $$|Df(x)k|^2 =\left<\nabla \rho (x) , k \right>^2+\rho^2 | DS(x) k|^2+2\left<\nabla \rho (x) , k \right> \left<S(x), DS(x)e_i\right>.$$

Since $|S(x)|^2=1$, we have $\left<S(x), DS(x)k\right>=0$. Thus

\begin{equation}\label{alp}|Df(x)k|^2 =\left<\nabla \rho (x) , k \right>^2+\rho^2 |DS(x) k|^2.\end{equation}
So summing for $k=e_i$, and  $i=1,\dots,n$ we get
\begin{equation}\label{normq}\norm Df(x)\norm ^2=|\nabla\rho(x)|^2+\rho^2\norm DS\norm ^2.\end{equation}

Let $f$ be a function between $A$ and $B$. By $N(y,f)$ we denote
the cardinal number of $f^{-1}({y})$ if the last set is finite and
we set $N(y,f)=+\infty$ in the other case. The function $y\to
N(y,f)$ is defined on $B$. If $f$ is surjective then $N(y,f)\ge 1$
for every $y\in B$. The following proposition hold.
\begin{proposition}\label{jakobi}\cite{RR}
Let $U$ be an open subset of $\mathbf{ R}^n$ and let $f:U\to \mathbf{ R}^n$
be $C^1$ mapping. Then the function $y\to N(y,f)$ is measurable on
$\mathbf{ R}^n$ and
\begin{equation}\label{jak} \int_{\mathbf{R}^n}N(y,f)\,\mathrm{d}y=\int_U|J(x,f)|\,\mathrm{d}x,\end{equation}
where $J(x,f)$ is the Jacobian of $f$.
\end{proposition}
Further, let $h$ be a $C^1$ surjection from an $n-1$ dimensional
rectangle $K^{n-1}$ onto the unit sphere $\mathbb{S}^{n-1}$. Let the
function $f$ be defined in the $n$ dimensional rectangle
$K^n=[0,1]\times K^{n-1}$ by $f(t,u)=rh(u)$. Thus $f$ is a $C^1$
surjection from $K^n$ onto the unit ball $\mathbb{B}^n$. It is easy to
obtain the formula $J(x,f)=t^{n-1}D_h(u)$, where $x=(t,u)\in
K^{n}$, and $D_h$ denotes the norm of the vector product
$$D_h=\left| \frac{\partial h}{\partial
x_1}\times\dots\times\frac{\partial h}{\partial
x_{n-1}}\right|.$$

According to Proposition~\ref{jakobi} it follows that
\begin{equation*}
\begin{split}
\dfrac{1}{n}\omega_{n-1}&=\mu (\mathbb{B}^n)=\int_{\mathbb{B}^n}\,\mathrm{d}y\le
\int_{\mathbb{B}^n}N(y,f)\,\mathrm{d}y\\&=\int_{K^n}|J(x,f)|\,\mathrm{d}x=
\int_0^1t^{n-1}\,\mathrm{d}t\int_{K^{n-1}}D_h(u)du=\dfrac{1}n\int_{K^{n-1}}D_h(u)du.
\end{split}
\end{equation*}
Consequently we have \begin{equation}\label{gram}
\int_{K^{n-1}}D_h(u)du\ge \omega_{n-1}.\end{equation}
Let $x\in \A(r,R)$ and define $N=\frac{x}{|x|}$. Then consider the following system of mutually orthogonal vectors $(U_1,\dots, U_{n-1},N)$ of the unit norm. The vectors $(U_1,\dots, U_{n-1})$ are arbitrarily chosen. Then we define the Gram determinant of $S$ at $x$ by $$D_S(x)=|D_{U_1}S(x)\times \dots\times D_{U_{n-1}}S(x)|.$$
Now we have the following refined version of \cite[Proposition~1.6]{Ka}.
\begin{lemma}\label{onep1}
Let $f$ be a $C^1$ surjection between the spherical rings
$\A(r,R)$ and $\A(r_\ast,R_\ast)$, and let $S={f}/{|f|}$. Let
$\mathbb{S}(t)$ be a sphere of radius $t$ centered at the origin. Then
\begin{equation}\label{inequ}\int_{\mathbb{S}(t)}D_S(x)d\sigma(x)\geq
\omega_{n-1}.
\end{equation} where $\omega_{n-1}$ denote the measure of $\mathbb{S}$.
\end{lemma}
\begin{proof}
Let $K^{n-1}$ be an $n-1$-dimensional rectangle and let
$g:K^{n-1}\to P^{n-1}$ be the spherical coordinates of $\mathbb{S}(t)$. Then the
function $S \circ g$ is a differentiable surjection from $K^{n-1}$
onto the unit sphere $\mathbb{S}$. Then by (\ref{gram}) we have $$
\int_{K^{n-1}}D_{S\circ g}dK\geq \omega_{n-1}.$$ Further we obtain $${D_{S\circ
g}(x)}={\left|S'(g(x))\frac{\partial g(x)}{\partial x_1}\times
\dots \times S'(g(x))\frac{\partial g(x)}{\partial
x_{n-1}}\right|}.$$ Hence we obtain
$$\omega_{n-1}\leq
\int_{K^{n-1}}D_S(g(x))D_g(x)dK(x)=
\int_{\mathbb{S}}D_S(\zeta)d\sigma(\zeta).$$ Thus we have
proved (\ref{inequ}).
\end{proof}

\section{The proof of the main result}
First we prove the following  corollary of Lemma~\ref{onep1}
\begin{lemma}\label{onep}
Let $f$ be a $C^1$ homeomorphism between the spherical rings
$\A(r,R)$ and $\A(r_\ast,R_\ast)$ in $\mathbf{R}^3$, and let $S={f}/{|f|}$. Let
$\mathbb{S}(t)$  be the sphere centered at $0$ with the radius $t\in(r,R)$. Then
\begin{equation}\label{inequ1}\int_{\mathbb{S}(t)}\left(\norm DS\norm ^2-\left|DS(x)\frac{x}{|x|}\right|^2\right)d\sigma(x)\geq
8\pi. \end{equation} The inequality \eqref{inequ1} is sharp and is attained for the mappings of the form $f(x)=\rho(|x|) T\left(\frac{x}{|x|}\right)$, where $T$ is an arbitrary conformal transformation of the $2-$sphere $\mathbb{S}$.

\end{lemma}
\begin{proof}

For fixed $x\in \A(r,R)$ let $N=\frac{x}{|x|}$ and assume that $U$, $V$ and $N$ is a system of mutually orthogonal vectors of the unit norm. Then $$\norm DS\norm^2 =|D_U S|^2+|D_V S|^2+|D_N S|^2$$ and so \begin{equation}\label{onep11}\norm DS\norm ^2-|D_N S|^2=|D_U S|^2+|D_V S|^2\ge 2|D_U S\times D_V S|=2 D_S. \end{equation} By integrating in $\mathbb{S}(t)$ and using Lemma~\ref{onep} we get \eqref{inequ1}.

Further if $T$ is a conformal mapping of $\mathbb{S}$ onto itself, and $f(x)=\rho(x)T\left(\frac{x}{|x|}\right)$ then the mapping $S:\mathbb{S}(t)\onto \mathbb{S}$ defined by $S(x)=T\left(\frac{x}{|x|}\right)$ is a conformal diffeomorphism between $\mathbb{S}(t)$ and $\mathbb{S}$. Moreover \begin{equation}\label{spli}\begin{split}2D_S(x)&=2|D_US(x)\times D_V S(x)|=|D_U S|^2+|D_V S|^2\\&=\norm DS\norm^2=\norm DS\norm^2-\left|DS(x)\frac{x}{|x|}\right|^2.\end{split}\end{equation} Thus $$\int_{\mathbb{S}(t)}\norm DS\norm ^2d\sigma(\eta)=8\pi.$$

\end{proof}

\begin{proof}[Proof of Theorem~\ref{krye}]
Before we go to the detailed proof let us make one shortcut. For every constant $a>0$ we have \begin{equation}\label{finv}\mathscr{F}\left[\frac{af}{|f|^2}\right]=\mathscr{F}[f].\end{equation}

In order to prove this statement, by calculations we find that for $g=\frac{af}{|f|^2}$ we have $$g_{x_i}=\frac{af_{x_i}}{|f|^2}-\frac{2af\left<f,f_{x_i}\right>}{|f|^4}, \ \ i=1,\dots, n.$$ Thus we obtain $$|g_{x_i}|^2=a^2\frac{|f_{x_i}|^2}{|f|^4},\ \ i=1,\dots, n.$$ Summing the previous inequalities we get $${\norm Dg\norm^2}=a^2\frac{\norm Df \norm^2}{|f|^4}.$$
It follows that $$\frac{\norm Df\norm^2}{|f|^2}=\frac{\norm Dg\norm^2}{|g|^2}.$$ This implies \eqref{finv}.

Thus  we can assume that $f$ maps the inner boundary onto the inner boundary and the  outer boundary onto the outer boundary, that means the following: $$\lim_{|x|\to r} |f(x)|=r_\ast$$ and $$\lim_{|x|\to R} |f(x)|=R_\ast.$$   By \eqref{normq} and Fubini's theorem we have
\[\begin{split}\mathscr{F}[f]&=\int_{\A(r,R)}\left(\frac{|\nabla \rho|^2}{\rho^2} + \norm D S\norm ^2\right)dx\\&=\int_r^R dt \int_{\mathbb{S}(t)}\left(\frac{|\nabla \rho|^2}{\rho^2} + \norm D S\norm ^2 \right)d\sigma(\eta)\end{split}\]
For fixed $\eta$, consider the curve $$\alpha(t)=f(t\eta)=\rho(t\eta)S(t\eta).$$

Then we have $$|\alpha'(t)|^2=|f'(t\eta)\eta|^2$$ and $|\alpha(r)|=r_\ast$ and $|\alpha(R)|=R_\ast$.

So $$|\alpha'(t)|^2=\left<\nabla \rho (t\eta) , \eta \right>^2+\rho^2(t\eta) |DS(t\eta) \eta|^2$$

Moreover \begin{equation}\label{pr}|\nabla \rho|^2\ge \left<\nabla \rho (t\eta) , \eta \right>^2=|\alpha'(t)|^2-\rho^2(t\eta) |DS(t\eta) \eta|^2\end{equation}

So \begin{equation}\begin{split}A&\ge 4\pi \int_r^R t^2 dt \frac{|\alpha'(t)|^2}{\alpha^2(t)} dt+\int_r^R t^2\int_{\mathbb{S}}\norm DS(t\eta)\norm ^2-|DS(t\eta) \eta|^2d\sigma(\eta)\\&=4\pi \int_r^R t^2 dt \frac{|\alpha'(t)|^2}{\alpha^2(t)} dt+\int_r^R\int_{\mathbb{S}(t)}\left(\norm DS(\zeta)\norm ^2-\left|DS(\zeta) \frac{\zeta}{t}\right|^2\right) d\sigma(\zeta)\end{split} \end{equation}

Further from \eqref{onep11} we have  \begin{equation}\label{duep}\int_{\mathbb{S}(t)}\left(\norm DS(\zeta)\norm ^2-\left|DS(\zeta) \frac{\zeta}{t}\right|^2\right) d\sigma(\zeta)\ge 8\pi. \end{equation}
Therefore \begin{equation}\label{threep}A\ge 4\pi \int_r^R \left(t^2\frac{|\alpha'(t)|^2}{|\alpha(t)|^2} +2\right)dt\ge \int_r^R \left(t^2\frac{(|\alpha(t)|')^2}{|\alpha(t)|^2} +2\right)dt.\end{equation}
If $f(x)=H(t)T(\frac{x}{|x|})$ then in view of \eqref{alp} and \eqref{spli} we have
$$\mathscr{F}[f]= \mathscr{H}[H]=4\pi\int_r^R \left(\frac{t^2\dot H^2}{H^2}+2\right) dt= 4\pi \int_r^R \left(t^2\frac{(|\alpha(t)|')^2}{|\alpha(t)|^2} +2\right)dt,$$ where $\alpha(t)=\rho(t\eta)\eta$, and $\eta$ is any fixed vector.
The Euler-Lagrange equation for the energy integral $\mathscr{H}$, as in \eqref{bubi} reduces to
\begin{equation}\label{sala} 2 \dot H(t) {H(t)}-{t \dot H(t)^2}+t \ddot H(t) {H(t)}=0.\end{equation} By taking the substitution $H(t)=\exp(K(t))$ in \eqref{sala} we arrive to the differential equation $$e^{K(t)} \left(2 \dot K(t)+t \ddot K(t)\right)=0 $$ whose general solution is $$K(t)=c_1+\frac{c_2}{t}.$$ Thus the general solution of \eqref{sala} is $$H(t)=ae^{b/t}, \ \   a>0,  \ \  b\in \mathbf{R}.$$
The diffeomorphisms $$H_1(t)=r_\ast \left(\frac{R_\ast }{r_\ast }\right)^{\frac{R (t-r)}{(R-r) t}}$$
and
$$H_2(t)=\frac{R_\ast r_\ast} {H_1(t)}={R_\ast} \left(\frac{r_\ast }{R_\ast }\right)^{\frac{R (r-t)}{(R-r) t}}$$
 map the interval $[r,R]$ onto $[r_\ast,R_\ast]$. The mapping $H_1$ preserves the orientation, and $H_2$ changes the orientation.
The  energy of this stationary mappings is $$\mathscr{F}[H_1]=\mathscr{F}[H_2]=4 \pi  \left(2(R-r)+\frac{r R \log \left[\frac{R_\ast}{r_\ast }\right]^2}{R-r}\right).$$

To prove that they are minimizers, we need to show that, we only need to
show that the given energy integral $$\mathscr{H}[H]=4\pi\int_r^R \left(\frac{t^2\dot H^2}{H^2}+2\right) dt$$ attains its minimum.

Define   $$\Lambda(t,H,\dot H)=\left(\frac{t^2\dot H^2}{H^2}+2\right),$$
and show that it is convex in $K=\dot H$.
For $K=\dot H$ we have the following formula
 \[\begin{split}\partial_{KK}\Lambda[t,H,K]=\frac{2t^2}{H^2}\end{split}\]
 which is clearly positive. Further since $r\le t\le R$ and $r_\ast\le H(t)\le R_\ast$, we can find a positive constant $C$ so that
\begin{equation}\label{bound1}C|\dot H|^2\le \Lambda[s,H,\dot H],\end{equation} which implies that the function $L$ is coercive.

  Let $H_m=H_m(t):[r.R]\to[r_\ast,R_\ast]$ be a sequence of smooth bijections  with $H_m(r)=r_\ast$, $H_m(R)=R_\ast$ and $$\inf_{H:[r.R]\to[r_\ast,R_\ast]}
  \mathscr{H}[H]=\lim_{m\to \infty} \mathscr{H}[H_m].$$ Then up to a subsequence it
   converges to a monotone increasing function $H_\circ$.  Moreover,  since $H_m$ is a bounded sequence of $\mathscr{W}^{1,2}$, it converges,
   up to a subsequence weakly to a mapping $H_\circ \in \mathscr{W}^{1,2}$.

 By using the convexity of $\mathcal{L}$  and the fact that $\mathcal{L}$ is coercive, by standard theorem from the calculus of
 variation (see \cite[p.~79]{bern}),  we obtain that $$ \mathscr{H}[H_\circ]= \lim_{m\to \infty}\ \mathscr{H}[H_m].$$

 Further as $\mathcal{L}[s,H,K]\in C^\infty(\mathbf{R}_+^3)$, with $\partial^2_{KK}\mathcal{L}[s,H,K]>0$, we infer that
 $H_\circ\in C^\infty[r,R]$ (see \cite[p.~17]{jost}) and $H_\circ$ is the solution of our Euler-Lagrange equation. Thus it coincides with $H_1$ or $H_2$.

To prove the equali statement, assume that in all inequalities \eqref{inequ1},\eqref{duep},\eqref{threep} is attained the equality. If \eqref{threep} is an equality, then $$\left<\alpha'(t),\frac{\alpha(t)}{|\alpha(t)|}\right> =|\alpha'(t)|$$ for every $t$. This implies that $\alpha'(t)=\varrho(t)\alpha(t)$ for $\varrho(t)>0$. Thus if $\alpha(t)=(x(t),y(t),z(t))$ we obtain $x(t)=c_1\exp(\int_r^t \varrho(t)dt)$, $y(t)=c_2\exp(\int_r^t \varrho(t)dt)$, $z=c_3\exp(\int_r^t \varrho(t)dt)$. In other words $\alpha$ is the part of the line $$\frac{x}{c_1}=\frac{y}{c_2}=\frac{z}{c_3},$$ orthogonal to the spheres that connect two points from the sphere. This means that if $\alpha(t)=\rho(t\eta)S(t\eta)$, then $S(t\eta)=S(\eta)$. In particular $D_N S=0$. If the equality is attained in  \eqref{inequ1}, then it is attained in \eqref{onep11}. Therefore $S_U \bot S_V$ that means the mapping $S_t(t\eta)=S(\eta)$ conformally maps $\mathbb{S}(t)$ onto $\mathbb{S}$ and does not depend on $t.$ Here the vectors $U$ and $V$ are mutually orthogonal and of unit norm (as in Lemma~\ref{onep}). If the equality is attained in \eqref{pr}, we get $$|\nabla \rho (x)|=|D_N \rho(x)|,$$ thus $D_U \rho(x)=0$ and $D_V \rho(x)=0$ which implies that $\rho(x)=\rho(|x|)$ (by abusing the notation). Thus we have proved that $u(x)=\rho(|x|)T\left(\frac{x}{|x|}\right)$ where $T$ is a conformal mapping of $\mathbb{S}$ onto itself.

\end{proof}

\section{Appendix}

It follows from Theorem~\ref{krye} that
\begin{corollary}\label{rrje}
Let $f\in \mathscr{W}^{1,2}$ be a homeomorphism between $\A(r,R)$ and $\A(r_\ast, R_\ast)$. Then \begin{equation}\label{notsh}\mathscr{E}[f]=\int_{\A(r,R)}\norm Df\norm^2 dx\ge \frac{4 \pi  }{R_\ast^2}\left(2(R-r)+\frac{r R \log \left[\frac{R_\ast}{r_\ast }\right]^2}{R-r}\right).\end{equation}
\end{corollary}
It seems that \eqref{notsh} is not sharp, but it shows that the minimizer of Dirichlet energy is not zero for the case of non-degenerated annuli. This is somehow complementary result to result for the case of degenerated annuli, where the infimum of the Dirichlet energy of Sobolev homomorphisms with free boundary condition is zero (\cite[Theorem~1.6]{advc}). It should be noticed the following, the solution to the equation $\Delta h=0$, if $h(x)=H(r)\frac{x}{|x|}$, according to \eqref{delta} is given by $$H(t)=ar +\frac{b}{t^2}.$$ Now the solution to the boundary value problem $$\left\{
                                                                                                       \begin{array}{ll}
                                                                                                         \Delta h=0, & \hbox{if $h=H(|x|)\frac{x}{|x|}$;} \\
                                                                                                         H(r)=r_\ast, H(R)=R_\ast, & \hbox{where $0<r<R$ and $0<r_\ast<R_\ast$,}
                                                                                                       \end{array}
                                                                                                     \right.$$ is given by $$H(t)=\frac{r^2 R^2 (-R r_\ast+r R_\ast)}{\left(r^3-R^3\right) t^2}+\frac{\left(r^2 r_\ast-R^2 R_\ast\right) t}{r^3-R^3}.$$
Then $$H'(t)=\frac{r^2 r_\ast-R^2 R_\ast}{r^3-R^3}+\frac{2 \left(r^2 R^3 r_\ast-r^3 R^2 R_\ast\right)}{\left(r^3-R^3\right) t^3}.$$
So $H'(t)>0$ for $t\in[r,R]$ if and only if $$\left(-2 r^2 R^3 r_\ast+2 r^3 R^2 R_\ast\right)+\left(-r^2 r_\ast+R^2 R_\ast\right) t^3\ge 0, \ \  t\in[r,R].$$
It follows that $$r^3 r_\ast+2 R^3 r_\ast-3 r R^2 R_\ast\le 0$$ i.e. the condition \begin{equation}\label{nit}\frac{r_\ast}{R_\ast}\le \frac{3 r R^2 }{r^3+2R^3},\end{equation} is sufficient and necessary for existence of radial Euclidean harmonic mappings between given annuli (the so-called generalized Nitsche condition).
In this case the harmonic mapping $h(x)=H(r)\frac{x}{|x|}$ satisfies the equation \begin{equation}\label{diri}\mathscr{E}[h]=\frac{4 \pi  \left(r \left(r^3+2 R^3\right) r_\ast^2-6 r^2 R^2 r_\ast R_\ast+R \left(2 r^3+R^3\right) R_\ast^2\right)}{R^3-r^3}.\end{equation}
It is clear that the quantity $X$ on the right hand side of \eqref{diri} is bigger than the quantity $Y$ on the right hand side of \eqref{notsh}. It is also clear that, $$Z=\inf\{\mathscr{E}[h]:h\in\mathscr{W}^{1,2}(\A(r,R),\A(r_\ast, R_\ast))\}\in (Y,X],$$ and probably $Z<X$, in view of (\cite[Theorem~1.6]{advc}), but the right value of $Z$ remains so far un-known.

\begin{conjecture}
Assume that $\mathcal{F}$ is the family a homeomorphisms between spherical rings $\A(r,R)$ and $\A(r_\ast, R_\ast)$ in $\mathbf{R}^n$ that belongs to $\mathscr{W}^{1,n-1}$. Then the Dirichlet integral of $f\in\mathcal{F}$ with respect to the weight $\wp(y)=|y|^{1-n}$ i.e. the integral $$\mathscr{F}[f]=\int_{\A(r,R)}\frac{\norm Df\norm ^{n-1}}{|f|^{n-1}} dx,$$   achieves its minimum for generalised-radial difeomorphisms between annuli.
\end{conjecture}

\end{document}